\documentclass{amsart}
\usepackage{amsfonts,amssymb,amscd,amsmath,enumerate,verbatim,calc}
\usepackage[all]{xy}

\newcommand{\CM}{Cohen-Macaulay}

\newcommand{\wrt}{with respect to}

\newcommand{\n}{\mathfrak{n} }
\newcommand{\m}{\mathfrak{m} }

\newcommand{\q}{\mathfrak{q} }
\newcommand{\rr}{\mathfrak{r} }

\newcommand{\bx}{\mathbf{x} }
\newcommand{\K}{\mathbf{K} }
\newcommand{\F}{\mathcal{F} }
\newcommand{\C}{\mathcal{C} }

\newcommand{\Z}{\mathbb{Z} }
\newcommand{\Q}{\mathbb{Q} }
\newcommand{\rt}{\rightarrow}

\newcommand{\ov}{\overline}

\newcommand{\gldim}{\operatorname{gldim}}
\newcommand{\End}{\operatorname{End}}

\newcommand{\pdim}{\operatorname{projdim}}
\newcommand{\rad}{\operatorname{rad}}

\newcommand{\htt}{\operatorname{ht}}
\newcommand{\Hom}{\operatorname{Hom}}

\newcommand{\Tor}{\operatorname{Tor}}

\theoremstyle{plain}

\newtheorem{theorem}{Theorem}[section]
\newtheorem{corollary}[theorem]{Corollary}
\newtheorem{lemma}[theorem]{Lemma}
\newtheorem{proposition}[theorem]{Proposition}

\theoremstyle{definition}
\newtheorem{definition}[theorem]{Definition}
\newtheorem{s}[theorem]{}

\theoremstyle{remark}

\begin{document}

\title[Crepant]{Examples of NON-COMMUTATIVE CREPANT RESOLUTIONS Of Cohen Macaulay normal  domains}
\author{Tony~J.~Puthenpurakal}
\date{\today}
\address{Department of Mathematics, IIT Bombay, Powai, Mumbai 400 076}

\email{tputhen@math.iitb.ac.in}
 \begin{abstract}
Let $A$ be a \CM \  normal   domain. A   non commutative crepant  resolution (NCCR) of $A$ is an $A$-algebra $\Gamma$ of the form $\Gamma = \End_A(M)$, where $M$ is a reflexive
$A$-module, $\Gamma$ is maximal \CM \  as an $A$-module and 
$\gldim(\Gamma)_P = \dim A_P $ for all primes $P$ of $A$.
We give bountiful examples of  equi-characteristic \CM \   normal local domains and  mixed characteristic \CM \   normal local domains  having NCCR. We also give plentiful examples of affine \CM  \  normal domains having NCCR.
\end{abstract}
 \maketitle
\section{introduction}
Let $A$ be a \CM \  normal  domain. Van den Bergh \cite{van} defined   \textit{a non-commutative crepant resolution }of $A$ (henceforth NCCR) to be  an $A$-algebra $\Gamma$ of the form $\Gamma = \End_A(M)$, where $M$ is a reflexive
$A$-module, $\Gamma$ is maximal \CM \  as an $A$-module and $\gldim(\Gamma)_P = \dim A_P $ for all primes $P$ of $A$.  We should remark that Van den Bergh only defined this for Gorenstein normal domains as this has applications in algebraic geometry. However there are many algebraic reasons for consider this generalization, see \cite{IW}. For a nice survey on this topic see \cite{Leu}.  In general,
it is subtle to construct NCCR's.  In this paper we give bountiful examples of \CM \ normal domains having a NCCR.

\s \label{Mix CM}\emph{Mixed Characteristic case:} We now outline in brief our construction. Recall $f \in \Z[X_1,\ldots,X_n]$ has content $1$ if $1$ belongs to the ideal generated by the coefficients of $f$. We say $f$ is $\Q$-smooth if $\Q[X_1,\ldots,X_n]/(f)$ is a regular ring. For a prime $p$ we say $f$ is smooth mod-$p$ if $\Z_p[X_1,\ldots,X_n]/(f)$ is a regular ring. It is well-known  that if $f$ is $\Q$-smooth then is smooth mod-$p$  for infinitely many primes $p$. Our result is:
\begin{theorem}\label{Ex-1}
Let $(A,\m)$ be an excellent normal \CM \  local domain of mixed characteristic with perfect residue field $k = A/\m$ of characteristic $p > 0$. 
Assume $A$ has a NCCR and that $\dim A \geq 2$.  Also assume that $A$ has a canonical module. Let $f \in \Z[X_1,\ldots,X_n]$ be of content $1$. Also assume that $f$ is $\Q$-smooth and is smooth mod-$p$. Set $T = A[X_1,\ldots,X_n]/(f)$ and let $\n$ be a maximal ideal of $T$ containing $\m T$. Set $A(f) = T_\n$.
Then
\begin{enumerate}[\rm (i)]
\item
$A(f)$ is flat over $A$ with regular fiber. In particular if $A$ is Gorenstein then so is $A(f)$.
\item
$A(f)$ is an excellent normal \CM \  local domain of mixed characterisitic  with perfect residue field.
\item
$A(f)$  has a  NCCR. 
\end{enumerate} 
Furthermore if $\Gamma = \Hom_A(M,M)$ is a  NCCR of $A$ then $\Lambda = \Gamma \otimes_A A(f)$ is a NCCR of $A(f)$. 
\end{theorem}

\s Two dimensional rings  of finite representation type have a  NCCR. 
(see \cite[Theorem-6]{Leu-E}).  For examples of two dimensional mixed characteristic rings of finite representation type see \cite{P}. Using the  above recipe  we can  construct plentiful examples of  \CM \ local domain of mixed characteristic having NCCR's. If $k$ is algebraically closed then it can be easily shown that if  $A(f) \cong A(g)$ as $A$-algebra's then the hypersurfaces defined by $f$ and $g$ in $\mathbb{A}^n(k)$ are birational

\s \label{equi}\emph{Equi-characteristic case (local):}
Let $(A,\m)$ be an excellent equi-characteristic \CM \  local domain with perfect residue field $k$. Assume $A$ contains $k$, $\dim A \geq 2$ and that it has a canonical module. Let $f \in k[X_1,\ldots,X_n]$ be smooth, i.e., $k[X_1,\ldots,X_n]/(f)$ is a regular ring.
We show
\begin{theorem}\label{Ex-2}
(with hypotheses as in \ref{equi})
Assume $A$ has a NCCR. Set $T = A[X_1,\ldots,X_n]/(f)$.  Let $\n$ be a maximal ideal of $T$ containing $\m T$. Set $A(f) = T_\n$.
Then
\begin{enumerate}[\rm (i)]
\item
$A(f)$ is flat over $A$ with regular fiber. In particular if $A$ is Gorenstein then so is $A(f)$.
\item
$A(f)$ is an equi-characteristic excellent normal \CM \  local domain  with perfect residue field.
\item
 $A(f)$ has  a  NCCR. 
\end{enumerate} 
Furthermore if $\Gamma = \Hom_A(M,M)$ is a NCCR of $A$ then $\Lambda = \Gamma \otimes_A A(f)$ is a NCCR of $A(f)$.
\end{theorem}

\s  In both Theorems  it is clear that
  $M \otimes_A  A(f)$ is reflexive and $\Lambda $ is maximal \CM \ as a $A(f)$-module.  To prove finiteness of global dimension of $\Lambda$ we may complete $A(f)$ (see \cite[1.4]{R}). The essential point is to prove the following result:
 
\begin{theorem}\label{main}
Let $(A,\m) \rt (B,\n)$ be a flat local homomorphism  of Henselian  local rings with fiber $F = B/\m B$ regular local. Assume the residue fields $k = A/\m$ and $l = B/\n$ are perfect.
Let $M$ be a finitely generated $A$-module such that $\Gamma = \Hom_A(M,M)$ has finite global dimension. Then 
$$\gldim \Gamma \otimes_A B   \leq  \gldim \Gamma  + \dim F.$$ 
\end{theorem} 

\s \label{global} \emph{Equi-characteristic case (global):} Let $k$ be a perfect field.  Let $A$ be an affine $k$-algebra. Assume $A$ is a \CM \  normal domain of dimension $d \geq 2$. Suppose $A$ has a NCCR $\Gamma = \Hom_A(M,M)$.  Let $f \in k[X_1,X_2,\ldots,X_n]$ be a polynomial such that its homogenization $\widetilde{f}$ defines a smooth hypersurface in $\mathbb{P}^{n}(\ov{k})$ (here $\ov{k}$ is the algebraic closure of $k$). We show
\begin{theorem}\label{Ex-3}
(with hypotheses as in \ref{global}) Let $T(f) = A[X_1,\ldots,X_n]/(f)$. Then
\begin{enumerate}[\rm (i)]
\item
$T(f)$ is a \CM \  normal domain. If $A$ is Gorenstein then so is $T(f)$.
\item
$\Gamma\otimes_A T(f)$ is a NCCR of $T(f)$.
\end{enumerate}
\end{theorem}
 Let $K$ be the quotient field of $A$. It is easily shown that if $T(f) \cong T(g)$ as $A$-algebra's  then $V(f)$ is isomorphic to $V(g)$ in the affine space $\mathbb{A}^n(\ov{K})$. Thus there are lot of examples of non-isomorphic rings $A(f)$ having NCCR.

\s The main technical tool in this paper is a notion we call absolutely indecomposable modules over a Hensel local ring. Let $(A,\m)$ be a Henselian  \ local ring of dimension $d \geq 0$ and residue field $k$. As $A$ is Henselian, the category of finitely generated $A$-modules is Krull-Schmidt, i.e., any finitely generated $A$-module is uniquely a finite direct sum of indecomposable $A$-modules.  Let $M$ be a finitely generated $A$-module and let $\rad \End_A(M)$ be the radical of $\End_A(M)$. 
Recall that a module $E$ is indecomposable if and only if $\End_A(E)$ is local; equivalently $\End_A(E)/\rad \End_A(E)$ is a division ring. We say $E$ is \textit{absolutely indecomposable} if $\End_A(E)/\rad \End_A(E) \cong k $. If $M = M_1^{a_1} \oplus \cdots \oplus M_n^{a_n}$ with $M_i$ absolutely indecomposable then \\
$\End_A(M)/\rad \End_A(M)$ is a direct product of matrix rings over $k$. This enables us to keep track of $\End_A(M)\otimes B$ when $B$ is flat over $A$. The main technical result of this paper is:
\begin{theorem}\label{split}
Let $(A,\m)$ be a Henselian local ring with perfect residue field $k$. Let $M$ be a finitely  generated $A$-module. Then there exists a finite flat extension of the form $R = A[X]/(\phi(X))$ where  $\phi(X)$ is monic and $\ov{\phi(X)}$ is irreducible in $k[X]$ such that the $R$-module $M\otimes_A R$ is a finite direct sum of absolutely indecomposable $R$-modules.  Furthermore 
$\gldim \End_A(M) = \gldim \End_R(M\otimes_A R)$. 
\end{theorem} 

We now describe in brief the contents of this paper. In section two we discuss some preliminaries that we need. In section three we introduce the notion of absolutely indecomposable modules. We prove Theorem 1.11 in section 4. In section five we give a description of $\End_A(E)/\rad \End_A(E)$. We prove Theorem 1.7 in section 6.  In section seven we prove Theorems 1.2 and 1.5. Finally in section eight we prove Theorem 1.9.
\section{preliminaries}
In this paper all commutative rings considered are Noetherian. Commutative rings will be denoted as $A, B$ etc.  All non-commutative rings considered will be an $A$-algebra for some commutative Noetherian ring $A$, furthermore they will be finitely generated as an $A$-module. Thus all non-commutative rings in this paper will be both left and right Noetherian. Non-commutative rings will be denoted as $\Gamma, \Lambda$ etc. Also all modules in this paper are left modules and they will be finitely generated.

In this section we collect some preliminaries which we need. I think that all the results here are already known. I include  proofs of some of them as I do not have a reference. 
\s \label{basic-0} Suppose $\Gamma$ is a ring finitely generated over $A$. Let $\rr = \rad \Gamma$. If $A$ is local  with maximal ideal $\m$ then 
$\Gamma / \rr$ is semisimple and $\rr \supseteq \m \Gamma \supseteq \rr^n$ for some $n \geq 1$. \cite[20.6]{Lam}.

The following is well-known. 
 \begin{proposition}\label{basic-1}
 Let $(A,\m)$ be local and let $M$ be an $A$-module. Set $\Gamma = \End_A(M)$ and  let $\rr = \rad \Gamma$. Then
\begin{enumerate}[\rm (1)]
\item
$\Hom_A(M,\m M)$ is a two sided ideal in $\Gamma$.
\item
$\Hom_A(M,\m M) \subseteq \rr.$
\item
$\m \Gamma \subseteq \Hom_A(M , \m M)$.
\end{enumerate} 
 \end{proposition}
An easy consequence of the above result is
\begin{proposition}\label{basic-2}
Let $f \colon (A,\m) \rt (B,\n)$ be a flat local map. Let $M$ be an $A$-module. Set 
$\Gamma = \End_A(M)$ and  let $\rr = \rad \Gamma$. Set $\Lambda  = \Gamma \otimes_A B$. Then $\rr \otimes B$ is a two sided ideal  contained in $ \rad \Lambda.$ 
\end{proposition}
\begin{proof}
Put $\q = \rr \otimes B$.  Clearly $\q$ is a two sided ideal of $\Lambda$. If we prove $\q^n \subseteq \rad \Lambda$ for some $n$ then we are done for $\Lambda/\rad \Lambda$ is semisimple.

By \ref{basic-0} we have that $\rr^n \subseteq \m \Gamma$ for some $n \geq 1$.
Let $\phi_1\otimes b_1, \phi_2\otimes b_2, \cdots, \phi_n \otimes b_n \in \q$. Set
\begin{align*}
\psi &= \phi_1\otimes b_1 \circ \phi_2\otimes b_2 \circ \cdots \circ \phi_n \otimes b_n, \\
    &= (\phi_1\circ\phi_2 \circ \cdots \circ \phi_n )\otimes (b_1b_2\cdots b_n)
\end{align*}
Thus $\psi = \phi \otimes b$ for some $\phi \in \rr^n$ and $b \in B$. 
As $\rr^n \subseteq \m \Gamma $ we get that $\phi(M) \subseteq \m M$. 
So $\psi(M\otimes B) \subseteq \m (M\otimes B)$. It follows that 
$\psi \in \Hom_B(M\otimes B, \n (M \otimes B)) \subseteq \rad \Lambda $ 
(by \ref{basic-1}). It follows that $\q^n \subseteq \rad \Lambda.$                                 
\end{proof}

A natural question is when $\rr\otimes B = \rad \Lambda$? We prove
\begin{lemma}\label{equal-rad}
Let $\phi \colon (A,\m) \rt (B,\n)$ be a flat local map with $\m B = \n$. Assume $k = A/\m$ is perfect. Let $M$ be an $A$-module. Set 
$\Gamma = \End_A(M)$ and  let $\rr = \rad \Gamma$. Set $\Lambda  = \Gamma \otimes_A B$. Then $\rr \otimes B = \rad \Lambda.$ 
\end{lemma}
\begin{proof}
By \ref{basic-2} we have that $\rr \otimes B \subseteq \rad \Lambda.$ It suffices to show that $\left(\Gamma/\rr \right)\otimes B$ is semisimple.  

By \ref{basic-0} we have that $\Gamma / \rr$ is semisimple. So
\[
\Gamma / \rr = M_{n_1}(\Xi_1)\times \cdots \times M_{n_r}(\Xi_r)
\]
where $\Xi_1,\ldots,\Xi_r$ are division algebras over $k$. Also note that $k \subseteq Z(\Gamma/ \rr)$, the center of $\Gamma/\rr$. It follows that $k \subseteq Z(\Xi_i)$ for each $i = 1,\ldots,r$. 

It suffices to show that $M_n(\Xi)\otimes B$ is semisimple where $\Xi$ is a division algebra finite dimensional over $k$ and $k \subseteq F = Z(\Xi)$. Set $l = B/\n$. As $k$ is perfect $F$ is separable over $k$ and so $F\otimes_k l = K_1\times \cdots \times K_s$ where $K_i$ are finite field extensions of $l$.
Notice
\begin{align*}
M_n(\Xi)\otimes_A B &= M_n(\Xi)\otimes_k k \otimes_A B \\
 &= M_n(\Xi)\otimes_k l \\
 &= M_n(\Xi)\otimes_F F \otimes_k l \\
  &= M_n(\Xi)\otimes_F (K_1\times K_2 \times \cdots \times K_s) \\
   &= (M_n(\Xi)\otimes_F K_1)\times (M_n(\Xi)\otimes_F K_2)\times \cdots  (M_n(\Xi)\otimes_F K_s).
\end{align*}
 It suffices to show that $M_n(\Xi)\otimes_F K$ is semisimple where $K$ is an extension of $F$. We first note that
 by \cite[15.1]{Lam}, the ring $\Upsilon = \Xi \otimes_F K$ is a simple ring.  Also note that $K = 1\otimes K$ is a subring of $\Upsilon$. As $\Xi$ is finite dimensional over $F$ we get that $\Upsilon$ is finite dimensional as a $K$-vector space. In particular $\Upsilon$ is Artinian. Thus by \cite[3.1]{Lam} $\Upsilon$ is a semisimple ring.   Finally notice that as $\Xi$ is finite dimensional over $F$, the natural ring homomorphism $M_n(\Xi)\otimes_F K \rt M_n(\Xi\otimes_F K)$ is an isomorphism, see \cite[7.4]{Lam}. The result follows. 
 \end{proof}
An easy consequence of the above result is the following:
\begin{corollary}\label{gdim-basic}
(with hypotheses as in \ref{equal-rad})
$$ \gldim \Gamma = \gldim \Lambda.$$
\end{corollary}
\begin{proof}
For a left $\Gamma $ module $M$ let $\pdim_\Gamma M$ denote its  projective dimension. 
By an argument similar to \cite[1.1]{R} we can show
\[
\pdim_\Gamma M = \pdim_\Lambda M\otimes_A B
\]
By \cite[1.3]{R} we have that $\gldim \Gamma = \pdim \Gamma/\rr$ (here we consider $\rr$ as a left $\Gamma$-ideal. Similarly   $\gldim \Lambda =  \pdim \Lambda/\rad \Lambda$. By \ref{equal-rad} we have that
\[
\frac{\Gamma}{\rr} \otimes_A B = \frac{\Lambda}{\rad \Lambda}.
\]
The result follows.
\end{proof}
\section{absolutely indecomposable modules}
Let $(A,\m)$ be a Henselian  \ local ring of dimension $d \geq 0$ and residue field $k$. Let $M$ be an  $A$-module and let $\rad \End_A(M)$ be the radical of $\End_A(M)$. 
Recall that a module $E$ is indecomposable if and only if $\End_A(E)$ is local; equivalently $\End_A(E)/\rad \End_A(E)$ is a division ring. We say $E$ is \textit{absolutely indecomposable} if $\End_A(E)/\rad \End_A(E) \cong k $. 

We need the following:
\begin{definition}
Suppose $(A,\m)$ is a Henselian  \ local ring of dimension $d \geq 0$ and residue field $k$.
Let $K \supseteq k$  be a field.  Then $\F(A,K)$ is the collection of Henselian local rings $(B,\n)$ such that 
 \begin{enumerate}
\item
\text{there is a flat local map}  $\phi \colon A \rt B$.
\item
$\m B = \n$.
\item
$B/\n \cong K$ over $k$.
\end{enumerate} 
\end{definition}
Note that by  \cite[App. Th\'{e}or\'{e}me 1, Corollaire]{BourIX} there exists a local ring $B$ (not-necessarily Henselian) satisfying (1), (2) and (3). Clearly  $\widehat{B} \in \F(A,K)$.

The word "absolutely indecomposable" is well chosen  thanks to the following result. 
\begin{theorem}
\label{abs-indec}
Let $(A,\m)$ be a Henselian local ring with perfect residue field $k$. Let $M$ be an $A$-module. The following are equivalent:
\begin{enumerate}[\rm (i)]
\item
$M$ is absolutely indecomposable.
\item
$M \otimes_A B$ is absolutely  indecomposable for every $B \in \F(A,K)$ for any extension field $K$ of $k$.
\item
$M \otimes_A B$ is indecomposable for every $B \in \F(A,K)$ for any extension field $K$ of $k$.
\item
$M \otimes_A B$ is indecomposable for some $B \in \F(A,K)$ with $K$ algebraically closed.
\end{enumerate}
\end{theorem}
\begin{proof}
Set $\End_A(M) = \Gamma$ and $\rr = \rad \Gamma$. For $B \in \F(A,K)$ set $\Lambda = \Gamma \otimes B$ and $\q = \rr \otimes B$.

$(i) \implies (ii)$. We have the exact sequence $0 \rt \rr \rt \Gamma \rt k$. Tensoring with $B$ yields
\begin{equation*}
0 \rt \q \rt \Lambda \rt K \rt 0. \tag{*}
\end{equation*}
By \ref{equal-rad} we have that $\q = \rad \Lambda$.
 It follows that  $M\otimes B$ is absolutely indecomposable. 

 $(ii) \implies (iii)$.  Clear.
 
 $(iii) \implies (iv)$. Clear.

$(iv) \implies (i)$.  Notice $\Lambda /\rad \Lambda$ is a divison algebra which is finite over $K$. As $K$ is algebraically closed we get that $\Lambda /\rad \Lambda = K$. 

 As $M\otimes B$ is indecomposable we have that $M$ is also indecomposable. Say $\Gamma /\rr = \Xi$ where $\Xi$ is a divison ring. Say  $\dim_k \Xi = r$. By \ref{equal-rad} we have that $\q = \rad \Lambda$. It follows that $\Xi\otimes_k K = K$. Computing dimensions  as vector space over $K$ we get that $r = 1$. So $\Xi = k$. Thus $M$ is absolutely indecomposable.
 
\end{proof}

\section{Proof of Theorem \ref{split} }
In this section we give a proof of Theorem \ref{split}. The essential ingredient is a construction which we now describe: 

\s \label{const} Let $(A,\m)$ be a Henselian local ring with perfect residue field $k$. Let $\ov{k}$ be the algebraic closure of $k$. Let  
$$\C_k = \{E \mid E \ \text{is a finite extension of } \ k, \ \text{and} \ E \subseteq \ov{k} \}.$$
Order $\C_k$ with the  inclusion as partial order. Note that $\C_k$ is a directed set, for if $E, F \in \C_k$ then the composite field $EF \in \C_k$ and clearly $EF \supseteq E$ and $EF \supseteq F$. We prove
\begin{theorem}
\label{basic-c}(with hypotheses as in \ref{const}) There exists a direct system of local rings $\{ (A^E,\m^E) \mid E \in \C_k\}$ such that
\begin{enumerate}[\rm (1)]
\item
$A^E$ is a finite flat extension with $\m A^E = \m^E$. Furthermore $A^E/\m^E  \cong E$ over $k$.
\item
$A^E$ is Henselian.
\item
For any $F, E \in \C_k$ with $F \subseteq E$ the maps in the direct system $\theta^E_F \colon A^F \rt A^E$ is flat and local with $\m^F A^E = \m^E$.
\end{enumerate}
\end{theorem}
The ring $T = \lim_{E\in \C_k} A^E$ will have nice properties  which enables us to prove Theorem \ref{split}.
\s \label{const-1}\textbf{Construction-1.1:}  For every  $E \in \C_k$  we construct a ring $A^E$ as follows.  As $k$ is perfect, $E$ is a separable extension of $k$. So by primitive element theorem  $E = k(\alpha_E)$ for some $\alpha_E \in E$. Let 
\[
p_E(X) = p_{E,\alpha_E}(X) = \mbox{Irr}(\alpha_E, k),
\]
be the unique monic minimal polynomial of $\alpha_E$ over $k$. Let $f_E(X) = f_{E, \alpha_E}(X)$ be a monic polynomial in $A[X]$ such that $\ov{f_E(X)} = p_E(X)$.
Set
\[
A^E = \frac{A[X]}{(f_E(X))}.
\]
Our construction of course depends on choice of $\alpha_E$ and the choice of $f_E(X)$. We will simply fix one choice of $\alpha_E$ and $f_E(X)$.
We prove:
\begin{proposition}\label{E-basic-1}
(with hypotheses as in \ref{const-1}) \begin{enumerate}[\rm (i)]
\item
$A^E$ is a finite flat extension of $A$.
\item
$A^E$ has a unique maximal ideal $\m^E$. Furthermore 
\begin{enumerate}[\rm (a)]
\item
$\m A^E = \m^E$.
\item
$A^E/\m^E \cong E$.
\end{enumerate}
\item
$A^E$ is a Henselian local ring.
\end{enumerate}
\end{proposition}
\begin{proof}
$(i)$ This is clear.\\
$(ii)$ Notice 
$$A^E/\m A^E = k[X]/(p_E(X)) \cong E.$$
It follows that $\m A^E$ is a maximal ideal of $A^E$. 
Also as $A^E$ is a finite extension of $A$ any maximal ideal $\n$ of $A^E$ will contain $\m$.
It  follows that $\m^E = \m A^E$ is the unique maximal ideal of $A^E$. Clearly  $(a), (b)$ hold.\\
$(iii)$. Let $S$ be a finite $A^E$-algebra. Then note that $S$ is a finite $A$-algebra. As $A$ is Henselian we get that $S$ is a product of local rings. Thus $A^E$ is Henselian.
\end{proof}

\s \label{const-2}\textbf{Construction-1.2:}
Let $k \subseteq F \subseteq E$ be a tower of fields. We construct a ring homomorphism $\theta^E_F \colon A^F \rt A^E$ as follows:
Notice $\alpha_F \in E$. It follows that $P_F(Y) = (Y - \alpha_F)h(Y)$ for some polynomial $h(Y) \in E[Y]$. As $F$ is separable over $k$, the roots of $P_F(Y)$ are all distinct. It follows that $(Y - \alpha_F, h(Y)) = 1$. The ring $A^E$ is Henselian. So there exists $\xi \in A^E$ and $g(Y) \in A^E[Y]$ such that $\ov{\xi} = \alpha_F$, $\ov{g(Y)} = h(Y)$ and $f_F(Y) = (Y-\xi)g(Y)$.\\
\textit{Claim-1:} If $\xi^\prime \in A^E$ such that $f_F(\xi^\prime) = 0$ and $\ov{\xi^\prime} = \alpha_F$ then $\xi^\prime = \xi$.\\
Note that $0 = f_F(\xi^\prime) = (\xi^\prime - \xi)g(\xi^\prime)$. As $h(\alpha_F) \neq 0$ we get that $g(\xi^\prime)$ is a unit in $A^E$. So $\xi^\prime  = \xi$. \\
\textbf{Notation:} Denote $\xi$ as $\xi^E_F$. \\
Define 
\begin{align*}
\theta^E_F \colon A^F &\rt A^E, \\
a &\rt a \ \text{for all} \ a \in A,\\
X &\rt \xi^E_F
\end{align*}
We prove:
\begin{proposition}\label{E-basic-2}
(with hypotheses as in \ref{const-2})
\begin{enumerate}[\rm (i)]
\item
$\theta^E_F$ is a homomorphism of $A$-algebra's.
\item
$\theta^E_F$ is a local map
and
$\m^FA^E = \m^E$.
\item
$A^E$ is a flat $A^F$-module (via $\theta^E_F$).
\item
If $k \subseteq F \subseteq  E \subseteq L$ is a tower of fields then we have a commutative diagram 
\[
\xymatrix{ 
A^F
\ar@{->}[d]_{\theta^E_F} 
\ar@{->}[dr]^{\theta^L_F}
 \\ 
A^E
\ar@{->}[r]_{\theta^L_E} 
& A^L
}
\]
\end{enumerate}
\end{proposition}
\begin{proof}
$(i)$ This is clear since $\theta^E_F(a) = a$ for each $a \in A$. \\
$(ii).$ As $\theta^E_F$ fixes $A$ we get that $\theta^E_F(\m_F) = \theta^E_F(\m A^F) \subseteq \m A^E = \m^E$. Thus $\theta^E_F$ is local. 

Also note that $\m_F A^E = \m A^F A^E = \m A^E = \m^E$.

$(iii).$  Suppose $\dim_F E = r$ and $\dim_k F = s$. Then $\dim_k E = rs$. Notice
$A^E/\m^F A^E = E$ it follows that $A^E$ is generated as an $A^F$ module by $r$-elements; say $\{ e_1,\ldots,e_r\}$. Similarly $A^F$ is generated as an $A$-module by $s$ elements; say $\{ f_1,\ldots, f_s \}$. It follows that $A^E$ is generated over $A$
by $\{ e_if_j \mid 1 \leq i \leq r, 1\leq j \leq s \}$. However $A^E$ is a free $A$-module of rank $\dim_k E = rs.$ It follows that $\{ e_if_j \}_{i,j}$ is a basis of the $A$-module $A^E$.\\
\textit{Claim:} $\{e_1,e_2,\ldots,e_r \}$ is a basis of $A^E$ over $A^F$. \\
We already have that $A^E$ is generated as an $A^F$ module by $\{e_1,e_2,\ldots,e_r \}$. Suppose
\[
\beta_1 e_1 + \beta_2 e_2 + \cdots + \beta_r e_r = 0 \quad \text{for some} \ \beta_i \in A^F.
\]
Write 
\[
\beta_i = \sum_{j = 1}^{s}\gamma_{ij}f_j \quad \text{for some} \ \gamma_{ij} \in A.
\]
It follows that
\[
\sum_{i,j} \gamma_{ij}e_if_j = 0
\]
As $\{ e_if_j\}_{i,j}$ is a basis of the $A$-module $A^E$ we get that $\gamma_{ij} = 0$ for all $i,j$. It follows that $\beta_i = 0$ for all $i$. Thus 
$\{e_1,e_2,\ldots,e_r \}$ is a basis of $A^E$ over $A^F$. 

$(iv)$. Note $f_F(Y) = (Y - \xi^E_F)g(Y)$ in $A^E[Y]$. Applying $\theta^E_F$ and noting that it fixes $A$ we get
\[
f_F(Y) = (Y - \theta^E_F(\xi^E_F))\theta(g(Y)) \quad \text{in} \ A^L[Y].
\]
Notice $\ov{\theta^E_F(\xi^E_F)} = \alpha_F$. So by uniqueness we get 
$$ \xi^L_F = \theta^E_F(\xi^E_F).$$
The result follows.
\end{proof}
As a consequence we get
\begin{proof}
[Proof of Theorem \ref{basic-c}] This follows from \ref{E-basic-1} and \ref{E-basic-2}.
\end{proof}

\s \label{const-3} \textbf{Construction-1.3:}
Set 
\[
T = \lim_{E \in \C_k} A^E,
\]
and let $\theta_E \colon E \rt T$ be the maps such that for any $F\subseteq E$ in $\C_k$ we have $\theta_E \circ \theta^E_F = \theta_F$.
For $F \in \C_k$ set
$$ \C_F = \{E \mid E \ \text{is a finite extension of } \ F \}.$$
Then clearly $\C_F$ is cofinal in $\C_k$. Thus we have
\[
T = \lim_{E \in \C_F} A^E.
\]
We have the following properties of $T$.
\begin{theorem}\label{prop-T}
(with hypotheses as in \ref{const-3})
\begin{enumerate}[\rm (i)]
\item
$T$ is a Noetherian ring.
\item
$T$ is a flat $A$-module. 
\item
$T$ is a flat $A^F$-module for any $F \in \C_k$.
\item
The map $\theta_E$ is injective for any $E \in \C_k$. 
\item
By (iv)
we may write $T = \bigcup_{E \in \C_k}A^E$. Set $\m^T = \bigcup_{E \in \C_k}\m^E$.
Then $\m^T$ is the unique maximal ideal of $T$. 
\item
 $\m T = \m^T$.
\item
$T/\m^T \cong \ov{k}$.
\item
$T$ is a Henselian ring.
\end{enumerate}
\end{theorem}
\begin{proof}
$(i)$. As $A^F \rt A^E$ is flat whenever $F \subseteq E$ and $\m^F A^E = \m^E$ we get that $T$ is Noetherian, see \cite[Chap.0, (10.3.13)]{EGA3}.

$(ii)$. Let $N$ be any $A$-module.
Notice for any $i \geq 1$ we have
\begin{align*}
\Tor^A_i(T,N) &= \Tor^A_i(\lim_{E \in C_k}A^E, N), \\
 &=\lim_{E \in C_k}\Tor^A_i( A^E, N) \\
 &= 0.
\end{align*}
Thus $T$ is a flat as an $A$-module.\\
$(iii)$. This is similar to $(ii)$.\\
$(iv)$. This follows since each map $\theta^E_F$ in the direct limit is injective.
\\
$(v)$. It is clear that $\m^T$ is an ideal in $T$. Suppose $\xi \notin \m^T$. Then $\xi \notin \m^E$ for some $E$. This implies that $\xi$ is a unit in $A^E$. So $\xi$ is a unit in $T$. Thus $\m^T$ is the unique maximal ideal of $T$.

$(vi)$. Clearly $\m T \subseteq \m^T$. Let $\xi \in \m^T$. Then $\xi \in \m^E$ for some $E \in \C_k$. But $\m^E = \m A^E$. It follows that $\xi \in \m T$. Thus $\m T  = \m^T$.

$(vii).$ The inclusion $\theta^E \colon A^E \rt T $ is flat local map of $A$-algebras and so induces an inclusion of fields $ \ov{\theta^E} \colon E \rt T/\m^T$ over $k$. It follows that $L = T/\m^T$ contains $\ov{k}$. Let $\ov{\xi} \in L$. Let $\xi \in T$ be its pre-image. Say 
$\xi \in A^E$. Then notice the map $\ov{\theta^E}$ maps $\xi + \m^E$ to $\ov{\xi}$. It follows that $L = \ov{k}$.

$(viii)$.  Let $f(Y) \in T[Y]$ be a monic polynomial such that its residue class $\ov{f(Y)}$ modulo $\m^T T[Y]$ has a factorization $\ov{f} = g^\prime h^\prime$ with monic polynomials $g^\prime, h^\prime \in T/\m^T[Y]$ and $(g^\prime, h^\prime) = 1$. 
By $(v)$ there exists $E \in \C_k$ such that $f(Y) \in A^E[Y]$. We may assume that all coefficients of $g^\prime, h^\prime \in F$ for some $F \in \C_k$. Set $K = EF$. Then note that $f \in A^K[Y]$ and modulo $\m^K K[Y]$ we have a factorization  $\ov{f} = g^\prime h^\prime$. As $A^K$ is Henselian we have that there exists monic polynomials 
$g,h \in A^K[Y]$ with $f = gh$ and $\ov{g} = g^\prime$ and $\ov{h} = h^\prime$. Now note that $g, h \in T[Y]$. 
\end{proof}

The significance of $T$ is that certain crucial properties descend to a finite extension $E$ of $k$.
\begin{lemma}\label{l-p-t}
(with hypotheses as above) 
\begin{enumerate}[\rm (1)]
\item
Let $M$ be a  $T$-module. Then there exists $E \in \C_k$ and an  $A^E$-module $N$ such that $M = N\otimes_{A^E} T$.
\item
Let $N_1,N_2$ be $A^E$-modules for some $E \in \C_k$. Suppose there is a $T$-linear map $f \colon N_1\otimes_{A^E} T \rt N_2 \otimes_{A^E} T$. Then there exists $K \in \C_k$ with $K \supseteq  E$ and an $A^K$-linear map $g \colon N_1\otimes_{A^E} A^K \rt N_2 \otimes_{A^E} A^K$ such that $f = g\otimes T$. Furthermore if $f$ is an isomorphism then so is $g$.
\end{enumerate}
\end{lemma}
\begin{proof}
$(1)$ Let $F_1 \xrightarrow{\phi} F_0 \rt M \rt 0$ be a finite presentation of $M$. Say $\phi = (a_{ij})$. Then by \ref{prop-T}-(v) there exists $E \in \C_k$ such that all $a_{ij} \in A^E$. Consider a presentation $ G_1\xrightarrow{\phi} G_0 \rt N \rt 0$
of $A^E$-module $N$. Clearly $N \otimes_{A^E}T \cong M$.

$(2)$. Notice $\Hom_{T}(N_1\otimes_{A^E}T, N_2\otimes_{A^E}T) \cong \Hom_{A^E}(N_1, N_2)\otimes_{A^E} T$. Thus 
\[
f = f_1\otimes \xi_1 +  \cdots + f_s\otimes \xi_s, \quad \text{for some} \ f_i \in \Hom_{A^E}(N_1, N_2) \ \text{and} \ \xi_i \in T.
\]
Then by \ref{prop-T}-(v) there exists $F \in \C_k$ such that all $\xi_i \in A^F$. 
Let $K = EF$. Set
$$g = f_1\otimes \xi_1 +  \cdots + f_s\otimes \xi_s \in \Hom_{A^E}(N_1, N_2)\otimes_{A^E} A^K.$$
Clearly $g \otimes T = f$. 

Let $U,V$ be the kernel and cokernel of $g$. If $f$ is an isomorphism then 
$U\otimes_{A^K} T = V \otimes_{A^K} T = 0$. By \ref{prop-T}-(iii) we have that $T$ is a faithfully flat extension of $A^K$. It follows that $U= V =0$. Thus $g$ is an isomorphism.
\end{proof}
We now give
\begin{proof}
[Proof of Theorem \ref{split}]
We make the construction as in \ref{basic-c}. Let $M\otimes_A T = L_1^{r_1} \otimes \cdots \oplus L_m^{r_m}$ where $L_1,\cdots,L_m$ are indecomposable $T$-modules. By \ref{l-p-t} it follows that there exists $E \in \C_k$ and $A^E$-modules $N_i$ with $N_i \otimes T = L_i$ for $i = 1,\ldots,m$. By \ref{abs-indec} it follows that $N_i$ are absolutely indecomposable. 
Notice
\[
M\otimes_A T = (M\otimes_A A^E) \otimes_{A^E} T \cong \left(\bigoplus_{i=1}^{s} N_i^{r_i} \right)\otimes_{A^E} T.
\] 
By \ref{l-p-t} there exists $K \in \C_k$ with $K \supseteq E$ such that
\[
M\otimes_{A^K} A^K = (M\otimes_{A} A^E)\otimes_{A^E} A^K \cong  \left(\bigoplus_{i=1}^{s} N_i^{r_i} \right)\otimes_{A^E} A^K =  \left(\bigoplus_{i=1}^{s} N_i^{r_i} \otimes_{A^E} A^K \right).
\]
By \ref{abs-indec} the $A^K$-modules $N_i \otimes_{A^E} A^K$ are absolutely indecomposable. We take $R = A^K$. Note that $R$ has the required form by \ref{const-1}. By \ref{gdim-basic} we get that $\gldim \End_A(M) = \gldim \End_R(M\otimes_A R)$.
\end{proof}
The following result will be useful later.
\begin{lemma}\label{uuu}
Let $\phi \colon (A,\m) \rt (B,\n)$ be a flat local map of Henselian rings. Assume the  residue fields $k,l$ of $A$ and $B$ are perfect. Let $M$ be an $A$ module. Then there exists a commutative diagram of flat local maps of Henselian rings
\[
\xymatrix{
A \ar@{->}[r]^{\eta}
     \ar@{->}[d]^\phi
&A^\prime 
     \ar@{->}[d]^{\phi^\prime}
\\
B \ar@{->}[r]^{\delta}
     &B^\prime      
}
\]
such that
\begin{enumerate}[\rm (1)]
\item
$(A^\prime,\m^\prime)$ is a finite flat extension of $A$ with $\m A^\prime = \m^\prime$. 
\item
$M\otimes A^\prime$ is a direct sum of absolutely indecomposable $A^\prime$-modules.
\item
 $(B^\prime,\n^\prime)$ is a finite flat extension of $B$ with $\n B^\prime = \n^\prime$. 
 \item
 If the fiber $F$ of $\phi$ is regular local then so is the fiber $F^\prime$ of $\phi^\prime$.
\end{enumerate}
\end{lemma}
\begin{proof}
$(1)$ and $(2)$: Let $A^\prime = A[X]/(f(X))$ be as in Theorem \ref{split}.

$(3)$. Set $\widetilde{B} = B\otimes A^\prime$. Then clearly $\widetilde{B}$ is a flat extension of $A^\prime$. Furthermore $\widetilde{B}$ is a flat extension of $B$. Also $\widetilde{B}$ is a finite extension of $B$. As $B$ is Henselian we get that $\widetilde{B}$ is a direct product of local rings say 
$\widetilde{B} = B_1\times \cdots \times B_s$. We note that as $k$ is perfect we get that 
\[
\widetilde{B}\otimes_B B/\n \cong k[X]/(f(X))\otimes_k l \cong K_1 \times \cdots \times K_m
\]
where $K_1,\cdots, K_m$ are finite field extensions of $l$. We now note that $B_1$ is localization of $\widetilde{B}$ at a maximal ideal. As $l$ is perfect we get that $B_1$ is a finite separable extension of $B$. So $\n B$ is the maximal ideal of $B_1$, see \cite[2.5]{W}. Set $B^\prime = B_1$.

$(4)$.  We note that $\delta$ induces a flat map $\ov{\delta} \colon B/\m B \rt B^\prime/\m B^\prime$. We note that as $\m A^\prime = \m^\prime$ we get that 
$\m B^\prime = \m^\prime B^\prime$. Thus $\ov{\delta} \colon F \rt F^\prime$ is a flat map. As $\n B^\prime = \n^\prime$ we get that the fiber of $\ov{\delta}$ is a field. So if $F$ is regular then so is $F^\prime$, see \cite[23.7]{Mat}.
\end{proof}
\section{A description of  $\End_A(E)/ \rad \End_A(E)$}
In this section $(A,\m)$ is a Henselian ring and $E$ is a finitely generated $A$-module. Assume
$E = E_1^{a_1}\oplus \cdots \oplus E_{s}^{a_s}$ where $E_1, E_2,\cdots, E_s$ are mutually non-isomorphic indecomposable $A$-modules. We prove
\begin{theorem}\label{formula}
(with hypotheses as above) 
$$\frac{\End_A(E)}{\rad \End_A(E)} \cong M_{a_1}\left(\frac{\End_A(E_1)}{\rad \End_A(E_1)} \right)  \times \cdots \times  M_{a_s}\left(\frac{\End_A(E_s)}{\rad \End_A(E_s)} \right)$$
(here for a ring $\Gamma$, we denote by $M_n(\Gamma)$ the ring of $n\times n$ matrices over $\Gamma$.
\end{theorem}
\s Let $M,N$ be two $A$-modules. Then note that we have an isomorphism of rings
\[
\End_A(M\oplus N) \cong \begin{bmatrix}
\End_A(M) & \Hom_A(N,M) \\ \Hom_A(M,N) & \End_A(N)
\end{bmatrix}
\]

Clearly Theorem \ref{formula} follows from the following:
\begin{theorem}
\label{for-basic} Let $M, N$ be two $A$-modules
and let $M = M_1^{a_1}\oplus \cdots \oplus M_s^{a_s}$ and $N = N_1^{b_1}\oplus \cdots \oplus N_r^{b_r}$ with $M_i,N_j$ indecomposable. Assume $M_i \ncong N_j$ for all $i,j$. Also assume $M_i \ncong M_j$ for $i \neq j$ and $N_i \ncong N_j$ for $i \neq j$. Then
$$ \rad \End_A(M\oplus N)  =  \begin{bmatrix}
\rad \End_A(M) & \Hom_A(N,M) \\ \Hom_A(M,N) &  \rad\End_A(N)
\end{bmatrix}$$
\end{theorem}
We need the following 
\begin{lemma}\label{comp}
(with hypotheses as in \ref{for-basic}) Let $f \in \Hom_A(M,N)$ and let \\ $g \in \Hom_A(N,M)$. Then
$g\circ f \in \rad \End_A(M)$ and $f \circ g \in \rad \End_A(N)$.
\end{lemma}
\begin{proof}
We induct on $s$, the number of factors of $M$.

We first prove when $s =1$. 
So $M = M_1^{a_1}$. 
It suffices to prove the result  when   $N$ is indecomposable.  If  $ N = D_1 \oplus D_2$ and assume the result is known for $D_1$ and 
$D_2$.  Let $f \colon M \rt N$ and $g \colon N \rt M $. Write $f = (f_1,f_2)$ and $g = g_1 + g_2$ where $f_i \colon M \rt D_i$ and $g_i \colon D_i \rt M$ for $i = 1,2$. Then $g\circ f = g_1\circ f_1 + g_2 \circ f_2$. By our assumption,  $g_i\circ f_i \in \rad \End(M)$. So $g\circ f \in \rad \End(M)$. 

We first consider the case when $a_1 =1$. So let $f \colon M_1 \rt N$ and let $g \colon N \rt M_1$. If $g\circ f \notin \rad \End(M_1)$ then it is invertible as $M_1$ is indecomposable.  It follows that $M_1 \cong N$ (see the proof in \cite[Chapter X, Lemma 7.6]{Lang})
This is a contradiction.

Now assume $a_1  \geq 2$. Let $f \colon M_1^{a_1} \rt N$ and let $ g \colon N \rt  M_1^{a_1}$. Write $f = f_1 + \cdots+ f_{a_1}$ and $g = [g_1,\cdots,g_{a_1}]^{tr}$ where $f_i \colon M_1 \rt N$ and $g_i \colon N \rt M_1$ for all $i$. Then notice $g\circ  f  = [g_if_j]$. By the previous case we get that $g_if_j \in \rad \End(M_1)$. It follows that
\[
g\circ f \in M_{a_1}(\rad \End(M_1)) =  \rad (M_{a_1}(\End(M_1)))  = \rad \End (M^{a_1}_1).
\] 
(for the first equality above see \cite[p.\ 61]{Lam}.)

Assume the result for $s = c$. We prove it when $s = c +1$. 
Let $M = M_1^{a_1}\oplus \cdots \oplus M_{c}^{a_c}\oplus M_{c+1}^{a_{c+1}}$. Set $D = M_1^{a_1}\oplus \cdots \oplus M_{c}^{a_c}$. 
Then $M = D \oplus M^{a_{c+1}}_{c+1}$. Let $f \colon M \rt N$ and $g \colon N \rt M$. Write $f = [f_1,f_2]$ and $g = [g_1,g_2]^{tr}$ where $f_1 \colon D \rt N$ and $f_2 \colon M_{c+1}^{a_{c+1}} \rt N$ and $g_1 \colon N \rt D$ and $g_2 \colon N \rt M_{c+1}^{a_{c+1}}$. 
Then
\[
g \circ f = \begin{bmatrix}   
g_1\circ f_1 & g_1\circ f_2 \\ g_2 \circ f_1 & g_2\circ f_2
\end{bmatrix}
\] 
By induction hypotheses we have $g_1\circ f_1 \in \rad \End(D)$ and $g_2\circ f_2 \in \rad \End(M_{c+1}^{a_{c+1}})$.
Set
\[
\xi_1 =  \begin{bmatrix}   
g_1\circ f_1 &  0 \\  0 & 0
\end{bmatrix} 
\quad \text{and} \ 
\xi_2 =  \begin{bmatrix}   
0 & g_1\circ f_2 \\  0 & 0
\end{bmatrix}
\]

\[
\xi_3 =  \begin{bmatrix}   
0 & 0 \\ g_2 \circ f_1 & 0
\end{bmatrix}
\quad \text{and} \ 
\xi_4 =  \begin{bmatrix}   
0 & 0 \\  0 & g_2\circ f_2
\end{bmatrix}
\]
As $g\circ f = \xi_1 + \xi_2 + \xi_3 + \xi_4$ it suffices to show that $\xi_i \in  \rad \End(M)$ for each $i$.

Let 
$\phi  = [\phi_{ij}] \in \End(M)$. Then notice
\[
1 - \phi \xi_1 = \begin{bmatrix} 1 - \phi_{11} \circ g_1 \circ f_1 & 0 \\ -\phi_{21}\circ g_1 \circ f_1 & 1 \end{bmatrix}
\] 
Notice $\phi_{11} \circ g_1 \circ f_1  \rad \End(D)$. So
$1-\phi_{11} \circ g_1 \circ f_1 $ is invertible in $\End(D)$. It follows that $1 - \phi \xi_1$ is invertible. So $\xi_1 \in \rad \End(M)$. Similarly $\xi_4 \in \rad \End(M)$.

We now prove $\xi_2 \in \rad \End(M)$.  Set $\theta =  g_1 \circ f_2$. Notice
\[
1 - \phi \xi_2 = \begin{bmatrix}  1 &    - \phi_{11} \circ \theta \\  0 & 1 - \phi_{21} \circ \theta\end{bmatrix}
\]
Also note that $\phi_{21} \colon D \rt M_{c+1}^{a_{c+1}}$ and  $\theta \colon M_{c+1}^{a_{c+1}} \rt D$. So by induction hypotheses we have that $\phi_{21}\circ\theta  \in \rad \End(M_{c+1}^{a_{c+1}})$. It follows that $1 - \phi \xi_2 $ is invertible in  $ \End(M)$. So $\xi_2 \in \rad(\End(M))$. Similarly $\xi_3 \in \rad \End(M)$.
\end{proof}

We now give
\begin{proof} [Proof of Theorem \ref{for-basic}]
Set
$$ \q = \begin{bmatrix}
\rad \End_A(M) & \Hom_A(N,M) \\ \Hom_A(M,N) &  \rad\End_A(N)
\end{bmatrix}$$
We first prove that $\q$ is a two sided ideal contained in $\rad \End(M\oplus N)$.
Let $\xi = [\xi_{ij}] \in \q$. Let $\phi =  [\phi_{ij}] \in \End(M)$. Then
\[
\phi \circ \xi = \begin{bmatrix}    \phi_{11} \xi_{11} + \phi_{12}\xi_{21} & \phi_{11}\xi_{12} + \phi_{12} \xi_{22} \\ \phi_{21} \xi_{11} + \phi_{22} \xi_{21} & \phi_{21}\xi_{12} + \phi_{22} \xi_{22} \end{bmatrix}
\]
By Lemma \ref{comp} we have that $\phi_{12}\xi_{21}  \in \rad \End(M)$ and
$ \phi_{21}\xi_{12} \in \rad \End(N)$. It follows that $\phi \circ \xi \in \q$. Similarly $\xi \circ \phi \in \q$. Therefore $\q$ is an ideal in $\End(M\oplus N)$.

We now show that $\q \subseteq \rad \End(M\oplus N)$. Let $\xi = [\xi_{ij}] \in \q$.
Set
\[
\xi_1 =  \begin{bmatrix}   
\xi_{11} &  0 \\  0 & 0
\end{bmatrix} 
\quad \text{and} \ 
\xi_2 =  \begin{bmatrix}   
0 & \xi_{12} \\  0 & 0
\end{bmatrix}
\]

\[
\xi_3 =  \begin{bmatrix}   
0 & 0 \\ \xi_{21} & 0
\end{bmatrix}
\quad \text{and} \ 
\xi_4 =  \begin{bmatrix}   
0 & 0 \\  0 & \xi_{22}
\end{bmatrix}
\]
As $ \xi = \xi_1 + \xi_2 +\xi_3 + \xi_4$. It suffices to show that each $\xi_i \in \rad \End(M\oplus N)$. This is similar to the proof in \ref{comp}. 

As $\End(M \oplus N)/\q $ is semi-simple and $\q \subseteq \rad \End(M\oplus N)$ it follows that $\q = \rad \End(M\oplus N)$
\end{proof}

\section{Proof of Theorem \ref{main}}
In this section we give give a proof of Theorem \ref{main}. We restate it for the convenience of the reader.
\begin{theorem}\label{main-2}
Let $(A,\m) \rt (B,\n)$ be a flat local homomorphism  of Henselian  local rings with fiber $F = B/\m B$ regular local. Assume the residue fields $k = A/\m$ and $l = B/\n$ are perfect.
Let $M$ be a finitely generated $A$-module such that $\Gamma = \Hom_A(M,M)$ has finite global dimension. Then 
$$\gldim \Gamma \otimes_A B   \leq  \gldim \Gamma  + \dim F.$$ 
\end{theorem} 
\begin{proof}
We first consider the case $M = M_1^{a_1}\oplus \cdots \oplus M_s^{a_s}$ where  $M_i$ are distinct absolutely indecomposable $A$-modules. 
Set $\rr = \rad \Gamma$ and $\q = \rad (\Gamma \otimes B)$. 
By \ref{basic-2} we get that $\rr \otimes B \subseteq \q$. 
 Let $x_1,\ldots x_c \in \n$ be such that their images in $F$ minimally generate the maximal ideal of $F$. By \ref{basic-0} we get that $(\mathbf{x})(\Gamma \otimes B) \subseteq \q$. 

By Theorem \ref{formula} we get that $\Gamma/\rr$ is a direct product of matrix rings over $k$. Therefore $(\Gamma/\rr) \otimes B$ is a direct product of matrix rings over $F$. It follows that $\q = \rr\otimes B  + (\mathbf{x})(\Gamma \otimes B)$.

By \cite[p.\ 177]{Mat}, $\bx$ is a $B$-regular sequence. Also note that $\bx \in Z(\Gamma \otimes B)$ the center of $\Gamma \otimes B$. Let $\K = \K(\bx, B)$ be the Koszul complex of $B$ \wrt \ $\bx$.  As $\Gamma/\rr \otimes B$ is a direct product of matrix rings over $F$ we get that 
$$\mathbf{C} = \K\otimes_B \left( \frac{\Gamma}{\rr} \otimes B \right)$$
is also acyclic with zeroth homology group $(\Gamma\otimes B)/\q$. 

We now note that $\pdim \mathbf{C}_i = \pdim \Gamma/\rr \otimes B$ for each module in the complex $\mathbf{C}$. Also note that 
$\pdim \Gamma/\rr \otimes B \leq \pdim \Gamma/\rr = \gldim \Gamma$, here the second equality holds since $\Gamma$ is semi-perfect with radical $\rr$. It follows that 
$\pdim  (\Gamma\otimes B)/\q \leq c + \gldim \Gamma$.  We note that $\Gamma \otimes B$ is semi-perfect with radical $\q$. 
So
\[
\gldim \Gamma\otimes B = \pdim (\Gamma\otimes B)/\q \leq c + \gldim \Gamma.
\]
Thus we have proved the result in this case. 

Now we consider the general case. By \ref{uuu} there exists a commutative diagram of flat local maps of Henselian rings
\[
\xymatrix{
A \ar@{->}[r]^{\eta}
     \ar@{->}[d]^\phi
&A^\prime 
     \ar@{->}[d]^{\phi^\prime}
\\
B \ar@{->}[r]^{\delta}
     &B^\prime      
}
\]
such that
\begin{enumerate}[\rm (1)]
\item
$(A^\prime,\m^\prime)$ is a finite flat extension of $A$ with $\m A^\prime = \m^\prime$. 
\item
$M\otimes A^\prime$ is a direct sum of absolutely indecomposable $A^\prime$-modules.
\item
 $(B^\prime,\n^\prime)$ is a finite flat extension of $B$ with $\n B^\prime = \n^\prime$. 
 \item
 The fiber $F^\prime$ of $\phi^\prime$ is also regular.
\end{enumerate}
Notice $\dim F^\prime = \dim F = c$.  Let $ t = \gldim \Gamma$. By \ref{gdim-basic} we have that $  \gldim \Gamma \otimes_A A^\prime = \gldim \Gamma = t.$ By our previous case we have that
\[
\gldim ( \Gamma \otimes_A A^\prime)\otimes_{A^\prime} B^\prime \leq  t + c.
\]
We now note that
\[
( \Gamma \otimes_A A^\prime)\otimes_{A^\prime} B^\prime  \cong \Gamma \otimes_A  B^\prime  \cong ( \Gamma \otimes_A B)\otimes_B B^\prime.
\]
By \ref{gdim-basic} we get that
\[
\gldim \Gamma \otimes_A B  =  \gldim ( \Gamma \otimes_A B)\otimes_B B^\prime \leq t + c.
\]
\end{proof}

\section{Proof of Theorems \ref{Ex-1} and \ref{Ex-2} }

We first give
\begin{proof}[Proof of Theorem \ref{Ex-1}]
We first note that if $L$ is a field of characteristic $0$ or $p$, the ring 
$L[X_1,\ldots,X_n]/(f)$ is regular.

(i) By \cite[p.\ 177]{Mat} we get that $T$ is a flat extension of $A$. So the map
$\phi \colon A \rt A(f)$ is  flat and local. Also note the fiber of $\phi$ is 
$k[X_1,\ldots,X_n]/(f)$ localized at a maximal ideal. In particular it is regular local. Thus $A(f)$ is  \CM. Furthermore $A(f)$ is Gorenstein if $A$ is.

(ii) Clearly $A(f)$ is excellent. The residue field of $A(f)$ is the residue field of $k[X_1,\ldots,X_n]/(f)$ localized at a maximal ideal. In particular it is a finite extension of $k$ and so it is perfect. It is also clear that $A(f)$ is of mixed characteristic. 

As $A(f)$ is local, to prove that it is a normal domain it suffices to prove it is normal. As $A(f)$ is \CM \ clearly it satisfies $S_2$. Also as $A(f)$ is catenary it suffices to show that $A(f)_P$ is regular for every prime ideal $P$ of height one.

Let  $P$ be a prime ideal in $A(f)$ of height one. Let $\q = P \cap A$. Note that we have a flat local map $\psi \colon A_\q \rt A(f)_P$. Let $F$ be the fiber of $\psi$. 
We now note that
\[
1 = \htt P = \dim A(f)_P = \dim A + \dim F = \htt \q + \dim F.
\]
Thus $\htt \q \leq 1$.

\textit{Case 1:} $\htt \q = 0$. So $\q =0$. Let $K$ be the quotient field of $A$. Then $A_\q = K$. Also $F$ is a localization of $K[X_1,\ldots,X_n]/(f)$ and so is regular. It follows that $A(f)_P$ is regular in this case.

\textit{Case 2:} $\htt \q = 1$. This implies that $\dim F = 0$. Let $\kappa(\q)$ be the 
the residue field of $A_\q$. Then note that $F$ is $\kappa(\q)[X_1,\ldots,X_n]/(f)$ localized at a minimal prime.  We note that $\kappa(\q)$ is either of characteristic zero or $p$. 
As observed earlier $\kappa(\q)[X_1,\ldots,X_n]/(f)$ is a regular ring. So $F$ is a field.  As $A$ is normal, $A_\q$ is regular. It follows that $A(f)_P$ is regular.

Thus $A(f)$ satisfies $R_1$. So $A(f)$ is normal. As $A(f)$ is local we get that $A(f)$ is a normal domain.  

(iii) Set $B = A(f)$.  Let $\Gamma = \End_A(M)$ be a NCCR of $A$. It is clear that $M\otimes_A B$ is a reflexive $A$-module and that $\Lambda = \Gamma \otimes B$ is maximal \CM \ as an $B$-module. We also note that $B$ has a canonical module, \cite[3.3.14]{BH} Thus it suffices to prove $\gldim \Lambda = \dim B$, \cite[2.17]{IW-I}. 
By \ref{gdim-basic} we may complete $B$. Thus it suffices to prove $\gldim \Lambda\otimes \widehat{B}  = \dim B$.  By \cite[22.4]{Mat}the map $\phi \colon A \rt  B$ extends to a flat map
$\widehat{\phi} \colon \widehat{A} \rt \widehat{B}$. Also note that
\[
\Lambda \otimes_B \widehat{B}  = (\Gamma\otimes_A \widehat{A})\otimes_{\widehat{A}} \widehat{B}.
\]
By Theorem \ref{main} we get that
\begin{align*}
\gldim \Lambda \otimes_B \widehat{B} &\leq \gldim \Gamma \otimes_A \widehat {A}  + \dim \widehat{B} - \dim \widehat{A} \\
&= \gldim \Gamma + \dim B - \dim A \\
&= \dim B
\end{align*}
Thus $\gldim \Lambda  \leq \dim B$, by 2.4. As $\Lambda$ is maximal \CM \ we always have $\gldim \Lambda \geq \dim B$, see \cite[F.1]{Leu}.
Thus $\gldim \Lambda = \dim B$. It follows that $\Lambda$ is a NCCR for $B$.
\end{proof}

\s  A proof of Theorem \ref{Ex-2} can be given along the same lines as above. The only thing to note that for any prime $\q$ in $A$ the residue field $\kappa(\q)$ of $A_\q$ contains $k$. So $\kappa(\q)[X_1,\ldots,X_n]/(f)$ is a regular ring.

\section{Proof of Theorem \ref{Ex-3}}
In this section we give
\begin{proof}[Proof of Theorem \ref{Ex-3}]
(i) Let $K$ be the quotient field of $A$.  We first prove that $f$ is irreducible in $K[X_1,\ldots,X_n]$.  Let $\ov{K}$ be the  algebraic closure of $K$. As $k \subseteq K$ we get that $\ov{k} \subseteq \ov{K}$. As $\widetilde{f}$ is smooth in $\mathbb{P}^n(\ov{k})$ we get that $\widetilde{f}$ and its partial derivatives do not have a common zero in $\mathbb{P}^n$.  Therefore $(X_0,X_1,\ldots,X_n) = \sqrt{(\widetilde{f}, J(\widetilde{f}))}$ in $\ov{k}[X_0,\ldots,X_n]$. It follows that $(X_0,X_1,\ldots,X_n) = \sqrt{(\widetilde{f}, J(\widetilde{f}))}$ in $\ov{K}[X_0,\ldots,X_n]$.  Therefore $\widetilde{f}$ is irreducible in $\ov{K}[X_0,X_1,\ldots,X_n]$. It follows that $f$ is irreducible in$\ov{K}[X_1,\ldots,X_n]$ and hence it is irreducible in $K[X_1,\ldots,X_n]$.

\textit{Claim 1:}  $T(f)$ is a domain.

We assert that $(f)K[X_1,\ldots,X_n] \cap A[X_1,\ldots,X_n] = (f)A[X_1,\ldots,X_n]$. If this assertion is proved then $T(f)$ will be a subring of $k[X_1,\ldots,X_n]/(f)$ and
so a domain.

Let $\xi \in (f)K[X_1,\ldots,X_n] \cap A[X_1,\ldots,X_n]$. Then $\xi = fg$ for some  \\ $g \in K[X_1,\ldots,X_n]$. Clearing denominators of $g$ we get that there exists $a \in A$ and $w \in
A[X_1,\ldots,X_n]$ such that $\xi a = f w$. We prove that $a$ divides all the coefficients of $w$. Let $P$ be a height one prime in $A$. Then $A_P$ is a DVR, so in particular a UFD.
Let  $R = A_P[X_1,\ldots,X_{n-1}]$ and let $L$ be the quotient field of $R$.  We may assume that  $X_n$ appears as a term in $f$. So $f \notin L$. Note $L$ is also the quotient field of $K[X_1,\ldots,X_{n-1}]$.  Note that by Gauss Lemma, $f$ is irreducible in $L[X_n]$ as the content of $f$ is one. So  again by Gauss Lemma we get that $f$ is irreducible in $R[X_n] = A_P[X_1,\ldots,X_n]$. It follows that if $c$ is a coefficient of $w$ then $(c/a)_P \in A_P$. But $A$ is a normal domain. So 
\[
\bigcap_{\htt  \q = 1} A_\q = A.
\] 
It follows that $c/a \in A$. Thus $\xi \in (f)A[X_1,\ldots,X_n]$.

By \cite[p.\ 177]{Mat}, $f$ is a non-zero divisor of $A[X_1,\ldots,X_n]$ and  $T(f)$ is a flat extension of $A$.  Also note that $\dim T(f) = \dim A + n -1$. Let $\n$ be a maximal ideal of $T(f)$. 

\textit{Claim-2} $\n \cap A$ is a maximal ideal of $A$. \\
Clearly $T(f)$ is an affine ring. Also by Claim-1 we have that $T(f)$ is a domain. Thus $\htt \n = \dim T(f) = \dim A + n -1$.  Let $P = \n \cap A$. We have a flat map 
$\psi \colon A_P \rt T(f)_\n$. Let $F$ be the fiber of $\psi$. Then note that $F$ is a localization of $\kappa(P)[X_1,\ldots,X_n]/(f)$.  As the content of $f$ is one we get that $f$ is a non-zero divisor in $\kappa(P)[X_1,\ldots,X_n]$. So $\dim F \leq n-1$. By the dimension formula for flat extensions it follows that $\htt P \geq \dim A$. So $P$ is a maximal ideal of $A$. 

\textit{Claim 3:} $T(f)$ is \CM. Also if $A$ is Gorenstein then so is $T(f)$.
 
 We first note that $k[X_1,\ldots,X_n]/(f)$ is a regular ring. 
 Let $\n$ be a maximal ideal of $T(f)$. Then  by Claim 2 we get that $\m = \n \cap A$ is a maximal ideal of $A$. We have a flat extension $\psi \colon A_\m \rt T(f)_\n$.  Note that $\kappa(\m) = A/\m$ is a finite extension of $k$. As $k$ is perfect  we get that $D = \kappa(\m)[X_1,\ldots,X_n]/(f)$ is regular. The fiber $F$ of $\psi$ is a localization of $D$ and so is regular. Thus $T(f)_\n$ is \CM \  and is Gorenstein if $A$ is. 

The assertion that $T(f)$ is normal follows exactly as in the argument in the proof of Theorem \ref{Ex-1}.

(ii) Let $\Gamma = \End_A(M)$ be a NCCR of $A$. Clearly $M\otimes T(f)$ is a reflexive $T(f)$-module. Furthermore $\Lambda = \Gamma\otimes_A T(f)$ is maximal \CM \ as a $T(f)$-module. The ring $T(f)$ has a canonical module \cite[18.21]{I-7}. So it suffices to prove that $\gldim \Gamma_\n = \dim T(f)_\n$ for all maximal ideals $\n$ of $T(f)$.
We note that $\n \cap A = \m$ a maximal ideal of $A$ (by claim 2). Also as $A$ and $T(f)$ are affine domains over $k$ we get that the residue fields of $\m$ and $\n$ are finite extensions of $k$ and so perfect. The proof of the assertion $\gldim \Gamma_\n = \dim T(f)_\n$ follows exactly as in the case of Theorem \ref{Ex-1}.
\end{proof}

\end{document}